\documentclass[preprint,12pt]{elsarticle}

\usepackage{multirow}
\usepackage{graphics}
\usepackage{tikz}

\usepackage[fleqn]{amsmath}

\usepackage{amsthm}
\usepackage{nccmath}
\usepackage{algorithm, algorithmic}
\usepackage{doi}
\usepackage{subcaption}

\newtheorem{theorem}{Theorem}[section]

\newtheorem{defn}[theorem]{Definition}

\newcommand{\diag}{\mathop{\mathrm{diag}}}  

\journal{Journal of Computational Physics}

\begin{document}

\begin{frontmatter}



\title{High-Resolution Solvers for 3D Helmholtz Scattering Problems Using PFFT and Eigenvector-Based Preconditioning}


\author[1]{Yury Gryazin}
\author[1]{Ron Gonzales}
\author[2]{Xiaoye Sherry Li}

\address[1]{Mathematics Department, Idaho State University, 921 S 8th Ave, STOP 8085, Pocatello, ID 83209-8085, USA.}
\address[2]{Lawrence Berkeley National Laboratory, Berkeley, CA, USA.}

\begin{abstract}
This paper presents an efficient Krylov subspace iterative solver for the three-dimensional (3D) Helmholtz equation with non-constant coefficients and absorbing boundary conditions, combining high-resolution compact schemes with low-order preconditioners. To mitigate numerical dispersion and reduce pollution error, we employ fourth- and sixth-order compact finite-difference schemes, thereby significantly softening the strict points-per-wavelength requirement. The resulting large, ill-conditioned linear systems are solved using a preconditioned GMRES method. The key innovation lies in the construction of the preconditioner: we introduce two highly efficient direct solvers—one based on a low-dimensional eigenvector transformation (EigT) and another on a partial Fast Fourier Transform (PFFT) algorithm—both derived from a lower-order approximation of the original problem that incorporates the absorbing boundary conditions. The motivation and efficacy of this lower-order preconditioning strategy for high-resolution schemes are analyzed through model problems, providing insight into the convergence rate. The theoretical analysis is validated by a comprehensive set of numerical experiments, demonstrating the method's performance for realistic problem sizes and parameters.
\end{abstract}

\begin{keyword}
Krylov subspace methods, Helmholtz equation, Eigenvectors transform, PFFT, GMRES, Compact finite-difference schemes, Scattering problems, Absorbing boundary conditions
\end{keyword}

\end{frontmatter}



\section{Introduction}

The pursuit of enhanced resolution in numerical algorithms remains a critical endeavor across many scientific and engineering disciplines. In numerous applications, the use of lower-order approximations for discretizing partial differential equations fails to yield results of sufficient accuracy on computationally feasible grid sizes. This demand for improved accuracy necessitates modifications to both the discretized systems and the numerical solvers themselves. In this paper, we develop a high-resolution, parallel approach for solving high-frequency three-dimensional scattering problems, leveraging novel lower-order preconditioners based on a Partial Fast Fourier Transform (PFFT) and Eigenvector Transform (EigT) methodology.

In our previous publications \cite{gkl, yg2, ggl}, we demonstrated the efficacy of this overarching methodology using standard FFT-type preconditioners derived from lower-order approximations. However, a significant limitation of conventional FFT solvers is their restriction to problems with Dirichlet, Neumann, or periodic boundary conditions within the preconditioning matrix. This inherent limitation creates a discrepancy between the boundary conditions of the high-resolution scheme and those of the preconditioner, leading to a substantial increase in the number of iterations required for convergence in the resulting iterative method.

This paper directly addresses this shortcoming by incorporating a lower-order approximation of the absorbing boundary conditions directly into the preconditioning matrix. The implementation of our proposed preconditioners involves solving two auxiliary one-dimensional eigenvalue problems, which facilitates a fast, parallel, direct solution of the 3D Helmholtz equation with constant coefficients and absorbing boundary conditions. This parallel algorithm provides an excellent and reliable preconditioner for medium-sized grids, extending to up to a thousand points in each direction of the three-dimensional domain.

The algorithm can be further enhanced by combining a 2D FFT solution (applicable under Dirichlet, Neumann, or periodic conditions) with a direct correction of the solution at the boundary points for each 2D slice of the 3D domain. This hybrid approach significantly accelerates the iterative solution on large grids, particularly those with sizes optimal for FFTs. For applications constrained by desktop computing resources and random-access memory limitations, the ability to obtain consistent results on arbitrary midsize grids presents a distinct advantage. Consequently, EigT-type algorithms present a valuable alternative to PFFT-type methods, which are more suited for large FFT-optimal grid sizes.

To illustrate the efficiency of the new GMRES solvers with PFFT-type and EigT-type preconditioners, we conduct a comprehensive comparison with previously developed FFT-preconditioned GMRES algorithms. Our numerical experiments demonstrate that the proposed preconditioned iterative solvers can outperform the standard GMRES-FFT combination \cite{gkl, yg2, eo} by close to an order of magnitude in the majority of test cases. We note that the solution of Helmholtz equations remains an actively developing research area, with significant recent contributions reflected in the following works: \cite{chen2021complex, galkowski2025numerical, lucido2024helmholtz, matsumoto2024fast, juraev2023applications, juraev2024helmholtz, klibanov2015two, kow2024minimization, yang2023novel}.

\subsection*{Problem Formulation and Methodology}
We consider a scattering problem formulated as the 3D Helmholtz equation:
\begin{equation}
\nabla ^{2}u + k^{2}(x,y,z)u = f(x,y,z), \quad \text{in } \Omega, \label{problem}
\end{equation}
subject to absorbing boundary conditions:
\begin{equation}
Bu = g, \quad \text{on } \partial \Omega, \label{bc}
\end{equation}
where $\Omega$ is a three-dimensional rectangular domain, $k(x,y,z)$ is a complex-valued variable coefficient, $\partial \Omega$ is the boundary of $\Omega$, and $B$ is a differential operator. In this paper we restrict our consideration to compact high-order approximations with Sommerfeld-like boundary conditions.

Improved resolution is achieved through fourth-order compact approximations (see, e.g., \cite{Lele92, sn}) and a recently developed sixth-order compact finite-difference scheme \cite{tggt}. The sixth-order scheme assumes a uniform grid in all three directions, while the fourth-order method imposes no such restriction.

The discretization yields a large, block 27-diagonal linear system. This matrix is generally neither positive definite nor Hermitian, causing most iterative methods to either fail to converge or converge impractically slowly. The most promising results for similar problems with constant coefficients and second-order schemes have been obtained using preconditioned Krylov subspace methods \cite{EG}.

This paper generalizes approaches developed for second-order discretizations of the Helmholtz equation \cite{EG, eo, eo1, gkl, yg2} to the case of compact fourth- and sixth-order approximation schemes with absorbing boundary conditions, using our novel EigT- and PFFT-type preconditioners.

\subsection*{Key Contributions and Novelty}
The pivotal element for rapid convergence is the choice of preconditioning matrix. Previous work \cite{EG, eo, gkl, yg2} used low-order preconditioning matrices that replaced radiation boundary conditions with Dirichlet or Neumann conditions on the domain's side faces. Systems with such matrices can be efficiently inverted using a combination of FFT and separation of variables in $O(N^3\log{N})$ operations. The attempt to improve the convergence of the method by applying a matching high-order approximation scheme \cite{yg2} with constant coefficient in the preconditioner but still with Dirichlet boundary conditions yields only marginal acceleration in convergence.

This paper demonstrates the critical importance of consistent boundary condition representation in the preconditioner for high-order schemes with variable coefficients, a phenomenon also noted in \cite{mp}. Our key innovation is to include the low-order approximation of the local boundary condition operator $B$ in the preconditioner. The resulting preconditioning matrix is inverted at each Krylov subspace step by a direct EigT or PFFT algorithm, extending methods proposed in \cite{hrt, TW} for higher-order schemes with non-constant coefficients. The outer iterations are performed using the Generalized Minimal Residual (GMRES) method \cite{s}.

The principal advantage of our approach is the consistent use of absorbing boundary conditions in both the high-resolution scheme and the lower-order preconditioner. The inherent parallelism of the low-order preconditioner ensures computational efficiency. The PFFT transform and inverse transform steps can be parallelized across subdomains separated by horizontal ($z=\text{const}$) planes, while the solution of independent tridiagonal systems can be parallelized by dividing the domain with vertical ($zx$- and $zy$-) planes. We plan to explore combinations of OpenMP, MPI, and CUDA implementations of these methods in future publications.

Numerical experiments on test problems demonstrate both the high accuracy of the high-order compact schemes and the computational efficiency of our preconditioned parallel iterative method.

Preliminary results of this work were presented at the 25th International Linear Algebra Society Conference (ILAS 2023) in Madrid, Spain, and at the 18th Copper Mountain Conference on Iterative Methods in Colorado, USA, in 2024.

\subsection*{Paper Organization}
The rest of this paper is organized as follows. Section 2 presents the fourth- and sixth-order compact 3D approximation methods. The convergence analysis, beginning with a motivating 1D model problem, is discussed in Section 3. Section 4 details the development of the parallel direct second-order EigT and PFFT preconditioners for these compact schemes. Finally, Section 5 demonstrates the effectiveness of the proposed algorithms on a series of test problems.

\section{Discretization}

\subsection{Statement of the Forward Problem}
In this section, we present the discretization of the 3D Helmholtz equation (\ref{problem}) with Sommerfeld-like boundary conditions on the computational domain's boundary:
\begin{eqnarray}
\frac{\partial u}{\partial n} - i k_{0} u = 0, \quad \text{on } \partial \Omega, \label{Sombc} \\
k_{0} = k(x,y,z)|_{(x,y,z) \in \partial \Omega}. \label{background}
\end{eqnarray}

In our numerical algorithm, we consider this problem on a finite rectangular computational domain $\Omega = [L_x^l, L_x^r] \times [L_y^l, L_y^r] \times [L_z^l, L_z^r]$ and impose the local boundary conditions (\ref{Sombc}) on the boundary $\partial \Omega$. In this paper, we focus on $k_{0} = \text{const}$, but the extension to a layered medium $k_{0} = k_{0}(z)$ is straightforward and will be considered in our future publications.

To develop fourth- and sixth-order compact approximation schemes for the solution of the 3D Helmholtz equation (\ref{problem}) with the boundary conditions (\ref{Sombc}), we consider a grid:
\begin{align*}
\Omega_h &= \{ (x_i, y_j, z_l) \mid x_i = L_x^l + (i-1) h_x,\ y_j = L_y^l + (j-1) h_y,\ \\
z_l &= L_z^l + (l-1) h_z, i = 1, \dots, N_x,\ j = 1, \dots, N_y,\ l = 1, \dots, N_z \},
\end{align*}
where $h_\alpha = (L_{\alpha}^r - L_{\alpha}^l) / (N_\alpha - 1)$ represents the grid step for each spatial direction $\alpha = x, y, z$. We employ the following standard central difference approximations for the first- and second-order derivatives at the grid point $(i,j,l)$:
\begin{eqnarray}
\delta_{x} u_{i,j,l} = \frac{u_{i+1,j,l} - u_{i-1,j,l}}{2h_x}, \quad \delta_x^2 u_{i,j,l} = \frac{u_{i+1,j,l} - 2u_{i,j,l} + u_{i-1,j,l}}{h_x^2}, \nonumber
\end{eqnarray}
where $u_{i,j,l} = u(x_i, y_j, z_l)$. Analogous definitions hold for the difference operators $\delta_{y}$, $\delta_{z}$, $\delta_{y}^2$, and $\delta_{z}^2$, which will be used in the following sections.

\subsection{A Fourth-Order Approximation Compact Scheme}

This subsection presents a standard fourth-order finite difference compact approximation for equation (\ref{problem}) based on the methodology provided in \cite{sn}. The fourth-order approximation of the second derivative can be expressed as:
\begin{equation}
\left . \frac{\partial^2 u}{\partial \alpha^2} \right |_{i,j,l} = \delta_\alpha^2 u_{i,j,l} - \frac{h_\alpha^2}{12} \left . \frac{\partial^4 u}{\partial \alpha^4} \right |_{i,j,l} + O(h_\alpha^4). \label{fourth_order_second_derivative}
\end{equation}
Differentiating equation (\ref{problem}) twice with respect to $x$ gives the following expression for the fourth derivative of $u$:
\begin{equation}
\frac{\partial^4 u}{\partial x^4} = \frac{\partial^2 f}{\partial x^2} - \frac{\partial^4 u}{\partial x^2 \partial y^2} - \frac{\partial^4 u}{\partial x^2 \partial z^2} - \frac{\partial^2 (k^2 u)}{\partial x^2}. \nonumber
\end{equation}
A second-order approximation of $\frac{\partial^4 u}{\partial x^4}$ can be written as:
\begin{equation}
\left . \frac{\partial^4 u}{\partial x^4} \right |_{i,j,l} = \delta_x^2 f_{i,j,l} - \delta_x^2 \delta_y^2 u_{i,j,l} - \delta_x^2 \delta_z^2 u_{i,j,l} - \delta_x^2 (k^2 u)_{i,j,l} + O(h_x^2). \label{second_order_fourth_derivative}
\end{equation}
Finally, by substituting (\ref{second_order_fourth_derivative}) into (\ref{fourth_order_second_derivative}), the fourth-order approximation of $\frac{\partial^2 u}{\partial x^2}$ becomes:
\begin{equation}
\left . \frac{\partial^2 u}{\partial x^2} \right |_{i,j,l} = \delta_x^2 u_{i,j,l} - \frac{h_x^2}{12} \left( \delta_x^2 f_{i,j,l} - \delta_x^2 \delta_y^2 u_{i,j,l} - \delta_x^2 \delta_z^2 u_{i,j,l} - \delta_x^2 (k^2 u)_{i,j,l} \right) + O(h_x^4). \nonumber
\end{equation}
The expressions for the fourth-order compact approximations of $\frac{\partial^2 u}{\partial y^2}$ and $\frac{\partial^2 u}{\partial z^2}$ have a similar form. Substituting these approximations into (\ref{problem}) and replacing $u_{i,j,l}$ with $U_{i,j,l}$, we obtain:
\begin{align}
&\left( \delta_x^2 + \delta_y^2 + \delta_z^2 \right) U_{i,j,l} + \frac{(h_x^2 + h_y^2)}{12} \delta_x^2 \delta_y^2 U_{i,j,l} + \frac{(h_x^2 + h_z^2)}{12} \delta_x^2 \delta_z^2 U_{i,j,l} \nonumber \\
&+ \frac{(h_y^2 + h_z^2)}{12} \delta_y^2 \delta_z^2 U_{i,j,l} + \left( 1 + \frac{h_x^2}{12} \delta_x^2 + \frac{h_y^2}{12} \delta_y^2 + \frac{h_z^2}{12} \delta_z^2 \right) (k_{i,j,l}^2 U_{i,j,l}) \label{scheme4} \\
&= \left( 1 + \frac{h_x^2}{12} \delta_x^2 + \frac{h_y^2}{12} \delta_y^2 + \frac{h_z^2}{12} \delta_z^2 \right) f_{i,j,l} = f^{(IV)}_{i,j,l}. \nonumber
\end{align}
Here, $U_{i,j,l}$ represents the fourth-order compact finite-difference approximation to $u_{i,j,l}$.

\subsection{A Sixth-Order Approximation Compact Scheme}

The compact sixth-order approximation scheme used in this paper was presented in \cite{tggt} for the three-dimensional Helmholtz equation with Dirichlet boundary conditions. The scheme requires uniform grid spacing, so we assume $h = h_x = h_y = h_z$. Using appropriate derivatives of (\ref{problem}), similarly to the previous subsection, the method can be written as:
\begin{align}
&( \delta_x^2 + \delta_y^2 + \delta_z^2 ) \left( 1 + \frac{k_{i,j,l}^2 h^2}{30} \right) U_{i,j,l} + \frac{h^4}{30} \delta_x^2 \delta_y^2 \delta_z^2 U_{i,j,l} + k_{i,j,l}^2 U_{i,j,l} \nonumber \\
&+ \frac{h^2}{6} ( \delta_x^2 \delta_y^2 + \delta_x^2 \delta_z^2 + \delta_y^2 \delta_z^2 ) \left( 1 + \frac{k_{i,j,l}^2 h^2}{15} \right) U_{i,j,l} + \frac{h^2}{20} \Delta_h (k_{i,j,l}^2 U_{i,j,l}) \label{scheme6} \\
&= f_{i,j,l} + \frac{h^2}{12} \nabla^2 f_{i,j,l} + \frac{h^4}{360} \nabla^4 f_{i,j,l} \nonumber \\
&+ \frac{h^4}{90} \left( \frac{\partial^4 f}{\partial x^2 \partial y^2} + \frac{\partial^4 f}{\partial x^2 \partial z^2} + \frac{\partial^4 f}{\partial y^2 \partial z^2} \right)_{i,j,l} = f^{(VI)}_{i,j,l}, \nonumber
\end{align}
where $f^{(VI)}_{i,j,l}$ represents the sixth-order modified right-hand side, and
\begin{align}
&\Delta_h (k_{i,j,l}^2 U_{i,j,l}) = k_{i,j,l}^2 f_{i,j,l} + \left( \nabla^2 (k^2) - k^4 \right)_{i,j,l}+ \nonumber \\
& 2[(k^2)_x]_{i,j,l} \left[ \delta_x \left( 1 + \frac{h^2}{6} ( \delta_y^2 + \delta_z^2 + k^2_{i,j,l}) \right) U_{i,j,l} - \frac{h^2}{6} (f_x)_{i,j,l} \right] +\nonumber \\
& 2[(k^2)_y]_{i,j,l} \left[ \delta_y \left( 1 + \frac{h^2}{6} ( \delta_x^2 + \delta_z^2 + k^2_{i,j,l}) \right) U_{i,j,l} - \frac{h^2}{6} (f_y)_{i,j,l} \right] +\nonumber \\
& 2[(k^2)_z]_{i,j,l} \left[ \delta_z \left( 1 + \frac{h^2}{6} ( \delta_x^2 + \delta_y^2 + k^2_{i,j,l}) \right) U_{i,j,l} - \frac{h^2}{6} (f_z)_{i,j,l} \right]. \nonumber
\end{align}

For the test problem with exact solutions satisfying boundary conditions (\ref{Sombc}), we used fourth and sixth-order approximations developed using an approach similar to the method described in \cite{tggt, yg2} with a three-point stencil.

But for scattering problems with non-constant $k$, the Sommerfeld-like boundary conditions (\ref{Sombc}) were approximated by placing the computational domain's boundary between the first two and last two grid points in each direction, using the central difference formulas:
\begin{align}
&\frac{U_{1,j,l} - U_{0,j,l}}{h} + i k_0 \frac{U_{1,j,l} + U_{0,j,l}}{2} = 0, \nonumber \\
&\frac{U_{N_x+1,j,l} - U_{N_x,j,l}}{h} - i k_0 \frac{U_{N_x+1,j,l} + U_{N_x,j,l}}{2} = 0, \nonumber \\
&j = 1, \dots, N_y,\ l = 1, \dots, N_z. \nonumber
\end{align}
These approximations allow elimination of the fictitious points $U_{0,j,l}$ and $U_{N_x+1,j,l}$ in the $x$-direction:
\begin{align}
U_{0,j,l} = \gamma_{x} U_{1,j,l}, \quad U_{N_x+1,j,l} = \gamma_{x} U_{N_x,j,l}, \quad j = 1, \dots, N_y,\ l = 1, \dots, N_z, \label{D_bc}
\end{align}
where $\gamma_x = (2 + i k_0 h_x) / (2 - i k_0 h_x)$. Similar expressions apply in the $y$- and $z$-directions with coefficients $\gamma_y$ and $\gamma_z$, respectively.

In some cases, we assume boundary conditions (\ref{Sombc}) are satisfied for any domain with boundaries $L_{\alpha}^l < L_{0,\alpha}^l$ and $L_{\alpha}^r > L_{0,\alpha}^r$, for some $L_{0,\alpha}^l$ and $L_{0,\alpha}^r$, where $\alpha = x, y, z$. Here, higher-order approximations to (\ref{Sombc}) can also be expressed in the form (\ref{D_bc}) with modified coefficients $\gamma_x$, $\gamma_y$, and $\gamma_z$ \cite{tggt}.

Having presented both fourth- and sixth-order compact schemes, we now turn to the solution strategy for the resulting linear systems. This paper uses low-order preconditioners combined with these high-order discretization matrices. The next section presents motivating theoretical results to justify this preconditioning approach, beginning with the simplified 1D case.

\section{Model Problems}

\subsection{One-Dimensional Problem}
In this section, we consider the Dirichlet boundary value problem for the 1D Helmholtz equation:

\begin{equation}
u_{xx} + k^2 u = f \quad \text{on } [0,1], \quad \text{with } u(0) = u(1) = 0. \label{1D_Helm}
\end{equation}
Let $u_i$ and $f_i$ be the values of the solution and the right-hand side at a grid point $x_i = ih$, for $i = 0,1,\dots,N+1$, where $h = 1/(N+1)$. The fourth-order approximation to the second derivative at the $i$-th grid point can be expressed as:
\begin{equation}
u_i'' = \delta_x^2 u_i - \frac{h^2}{12} u_i^{(4)} + O(h^4). \label{u_prime_app}
\end{equation}
We can now use the original equation to eliminate the fourth derivative in (\ref{u_prime_app}). Differentiating (\ref{1D_Helm}) twice yields:
\begin{equation}
u_i^{(4)} = -k^2 u_i'' + f_i''.
\end{equation}
Substituting this relation into (\ref{u_prime_app}) and setting $\tau = kh$, we obtain:
\begin{equation}
\left(1 - \frac{\tau^2}{12} \right) u_i'' = \delta_x^2 u_i - \frac{h^2}{12} f_i'' + O(h^4). \label{u_dprime}
\end{equation}
Let $U_i$, for $i = 0,1,\dots,N+1$, denote the compact fourth-order discrete approximation to the solution of (\ref{1D_Helm}). Then from (\ref{1D_Helm}) and (\ref{u_dprime}), we obtain the system:
\begin{eqnarray}
&& U_{i-1} - 2d U_{i} + U_{i+1} = F_i, \quad i = 1,\dots,N, \label{vect_eq} \\
&& U_0 = 0, \quad U_{N+1} = 0, \quad d = 1 - \frac{\tau^2}{2} + \frac{\tau^4}{24}, \nonumber \\
&& F_i = h^2 \left(1 - \frac{\tau^2}{12} \right) f_i + \frac{h^4}{12} f_i''. \nonumber
\end{eqnarray}
This system can be rewritten in the matrix form $A U = F$, where
\begin{equation}
A = \Lambda - 2d I \label{A_matrix}
\end{equation}
is a symmetric tridiagonal matrix, and
\begin{equation}
\Lambda = \begin{bmatrix}
0 & 1 & & \cdots & 0 \\
1 & 0 & 1 & & \vdots \\
& \ddots & \ddots & \ddots & \\
\vdots & & 1 & 0 & 1 \\
0 & \cdots & & 1 & 0
\end{bmatrix}. \label{lambda_mat_definition_new}
\end{equation}
For the fourth-order scheme defined above, we select a preconditioning matrix $A_p$ based on the second-order central difference approximation:
\begin{equation}
A_p = \Lambda - 2 \left(1 - \frac{\tau^2}{2} \right) I, \label{precond}
\end{equation}
so that the right-preconditioned system is:
\begin{equation}
A A_p^{-1} Y = F, \quad A_p U = Y. \label{Preconditioned_System}
\end{equation}
The above-described fourth-order scheme can be generalized to a scheme of any even order $2r$, $r \ge 2$. In this case, the coefficient $d = d_r$ and the right-hand side $F_i = F_{i,r}$ in equation (\ref{vect_eq}) are given by:
\begin{equation}
d_r = \sum_{j=0}^{r} (-1)^j \frac{\tau^{2j}}{(2j)!}, \label{coef_gen}
\end{equation}
and
\begin{equation}
F_{r,i} = 2h^2 \sum_{j=0}^{r-1} f^{(2j)}_i h^{2j} \sum_{l=0}^{r-j-1} (-1)^l \frac{\tau^{2l}}{[2(l+j+1)]!}, \quad i = 0,1,\dots,N. \label{rhs_gen}
\end{equation}
In what follows, we use the preconditioned GMRES method \cite{s} to find iterated approximate solutions $U^{(n)}$, $n \ge 0$, to the system (\ref{vect_eq}) generated by a compact scheme of order $2r$ with the second-order preconditioner (\ref{precond}). Our main goal is to estimate the rate of convergence to zero of the residuals $r_n = F - A U^{(n)}$. For a theoretical discussion of this approach, see \cite{yg2}. We first recall the definition of an $m$-th order preconditioned system introduced in \cite{yg2}.

Since the $N \times N$ matrices $A$ and $A_p$ share the same set of orthonormal eigenvectors $V$ as the matrix $\Lambda$ given in (\ref{LMB_egv}), the matrix $A A_p^{-1}$ can be represented as $A A_p^{-1} = V^{-1} B V$, where $B$ is the diagonal matrix of eigenvalues of $A A_p^{-1}$ and $V$ is the matrix of the corresponding eigenvectors.

\begin{defn}
A system in the form of (\ref{Preconditioned_System}) is said to be an $m$-th order preconditioned system if the $N \times N$ matrix $A A_p^{-1}$ is diagonalizable and can be expressed as $A A_p^{-1} = V^{-1} (I + h^m D) V$, where $0 < h < 1$ and $D = \text{diag}(d_{11}, \dots, d_{NN})$ is a diagonal matrix such that $\max_{1 \le i \le N} |d_{ii}| < M$ for some positive constant $M$ independent of $h$.
\label{Def}
\end{defn}

Using this definition, known estimates for the convergence rate of the GMRES method \cite{s} applied to the system (\ref{Preconditioned_System}) lead to the following result.
\begin{theorem}
Suppose that (\ref{Preconditioned_System}) is an $m$-th order preconditioned system. Then the iterations $U^{(n)}$ of the GMRES method applied to this system satisfy the convergence estimate:
\begin{equation}
\| r_n \|_2 \le \kappa_2(V) (M h^m)^n \| r_0 \|_2, \label{est}
\end{equation}
where $\kappa_2(V) = \| V^{-1} \|_2 \| V \|_2$ is the condition number of matrix $V$.
\label{theorem_gmres_estimate}
\end{theorem}
\noindent For a proof of Theorem \ref{theorem_gmres_estimate}, see \cite{yg2}.

We now apply the estimate (\ref{est}) to the solution of the system (\ref{Preconditioned_System}) for a discretization scheme of arbitrary even order $2r$.

\begin{theorem}
Let $N \times N$ matrices $A$ and $A_p$ be defined by (\ref{A_matrix}), (\ref{precond}), and (\ref{coef_gen}) with $r \ge 2$. Suppose there exists a constant $\delta_0 > 0$, independent of $h$, such that:
\begin{equation}
\left| \frac{4 \sin^2(i \pi h / 2)}{h^2} - k^2 \right| \ge k^2 \delta_0 \quad \text{for all } i = 1,\dots,N.
\end{equation}
Then (\ref{Preconditioned_System}) is a second-order preconditioned system, and the residuals satisfy the convergence estimate:
\begin{equation}
\| r_n \|_2 \le \left( \frac{2 \tau^2}{\delta_0} \sum_{l=0}^{r-2} \frac{k^{2l}}{4^l [2(l+2)]!} \right)^n \| r_0 \|_2. 
\end{equation}
\label{last_1D_theorem}
\end{theorem}
\begin{proof}
It is well-known (see, e.g., \cite{eo1, KressNA}) that the eigenvalues of matrix $\Lambda$ are $\lambda_i(\Lambda) = 2 \cos(\pi h i)$ for $i = 1,\dots,N$, and its orthonormal eigenvectors $V_i$, $i = 1,\dots,N$, are given by:
\begin{equation}
(V_i)_l = \sqrt{2h} \sin(\pi h i l), \quad l = 1,\dots,N. \label{LMB_egv}
\end{equation}
The matrix $V$ formed by these eigenvectors is unitary, so $\| V^{-1} \|_2 = \| V \|_2 = 1$.

From (\ref{A_matrix}), the eigenvalues of matrix $A$ are:
\begin{eqnarray}
\lambda_i(A) &=& \lambda_i(\Lambda) - 2d_r = -4 \sin^2 \left( \frac{\pi h i}{2} \right) + 2 \sum_{j=1}^{r} (-1)^{j-1} \frac{\tau^{2j}}{(2j)!}, \nonumber \\
i &=& 1,\dots,N. \label{EGV_A}
\end{eqnarray}
Similarly, setting $r = 1$ in (\ref{coef_gen}), we find from (\ref{precond}) that:
\begin{eqnarray}
\lambda_i(A_p) = -4 \sin^2 \left( \frac{\pi h i}{2} \right) + \tau^2, \quad i = 1,\dots,N. \nonumber
\end{eqnarray}
Therefore, for $i = 1,\dots,N$, we obtain:
\begin{eqnarray}
\lambda_i(A A_p^{-1}) &=& \frac{\lambda_i(A)}{\lambda_i(A_p)} = 1 + \frac{2 \sum_{j=2}^{r} (-1)^{j-1} \frac{\tau^{2j}}{(2j)!}}{-4 \sin^2(\frac{\pi h i}{2}) + \tau^2}, \nonumber \\
&=& 1 - 2 k^4 h^2 \frac{ \sum_{l=0}^{r-2} (-1)^{l+1} \frac{\tau^{2l}}{(2(l+2))!} }{ \frac{4 \sin^2(\frac{\pi h i}{2})}{h^2} - k^2 }. \label{1DEGV}
\end{eqnarray}
Setting:
\begin{eqnarray}
d_{ii} = -2 k^4 \frac{ \sum_{l=0}^{r-2} (-1)^{l+1} \frac{\tau^{2l}}{(2(l+2))!} }{ \frac{4 \sin^2(\frac{\pi h i}{2})}{h^2} - k^2 }, \quad i = 1,\dots,N, \label{Dii}
\end{eqnarray}
and using the assumptions of Theorem \ref{last_1D_theorem} and the inequality $0 < h < \frac{1}{2}$, we find that:
\begin{eqnarray}
|d_{ii}| < \frac{2 k^2}{\delta_0} \sum_{l=0}^{r-2} \frac{k^{2l}}{4^l (2(l+2))!} := M, \quad i = 1,\dots,N. \label{Dii_new}
\end{eqnarray}
Applying Definition \ref{Def} and formula (\ref{1DEGV}), we conclude that (\ref{Preconditioned_System}) is a second-order preconditioned system. Using Theorem \ref{theorem_gmres_estimate}, we obtain the required estimate for the residuals.
\end{proof}

These estimates demonstrate the main feature of the proposed algorithms: the number of iterations required to achieve a given accuracy decreases as the size of the system increases. This property is unexpected and generally does not hold for basic iterative methods.

For instructions on estimating $\delta_0$ and a proof for the sixth-order compact approximation scheme in one and three-dimensional cases, see \cite{yg2}.

\section{Efficient Preconditioners: EigT and PFFT Methods}

In our prior work \cite{yg2, ggl, gkl}, we developed a family of efficient high-resolution iterative methods combining preconditioned GMRES with a lower-order FFT preconditioner. The preconditioning matrix $A^F_p$ for this approach was derived using a second-order approximation central difference scheme with the absorbing boundary on the sides of the three-dimensional computational domain replaced with the Dirichlet or Neumann conditions. Moreover, the coefficient  $k(x,y,z)$ in the governing equation (\ref{problem}) was replaced with the background coefficient $k_{0}(z)$ (\ref{background}) in the preconditioner. While effective, the mismatch between the original matrix's absorbing boundary conditions and the preconditioner's Dirichlet/Neumann approximations limited convergence rates.

In this paper, we improve the previous approach by including a lower-order approximation of the absorption boundary conditions on all boundaries in the preconditioning matrix. This change requires modification of the transformation steps of the direct FFT preconditioner. The proposed method relies on a solution of a simple one-dimensional eigenvalue problem, which can be precomputed. Then, the transformation step can be implemented in parallel directly using two precomputed matrices of eigenvectors with sizes $N_x$ by $N_x$ and $N_y$ by $N_y$  or using the Partial Fast Fourier Transform (PFFT) approach. The theoretical asymptotic complexity of the direct eigenvector transformation for each two-dimensional slice is $O(N_x^2N_y+N_xN_y^2)$ instead of $O(N_xN_y\log(N_xN_y))$ required by the PFFT.
However, empirical results show comparable runtime for grids with $N_x, N_y < 10^3.$ We present the results of these experiments for different wavelengths and several grid sizes.    

\subsection{Low-Order Approximation Direct Eigenvectors (EigT) Preconditioner}
 
 First, we introduce a non-Hermitian tridiagonal matrix that represents the major building block of the second-order 3D preconditioning matrix:
\begin{equation}
\bar{\Lambda}_{\nu} = \begin{bmatrix}
\gamma_{\nu}& \zeta_{\nu} & & \cdots &0\\
1& 0 &1 &  &\vdots \\
& \ddots & \ddots&\ddots  & \\
\vdots&  & 1&  0&1 \\
0&  \cdots& &  \zeta_{\nu}& \gamma_{\nu} 
\end{bmatrix}_{N_{\nu}\times N_{\nu}}, \gamma_{\nu}, \zeta_{\nu} \in \mathbf{C} \text{ with } \nu = x, y, z,
\label{lambda_mat_definition}
\end{equation}
where $N_{\nu}$ represents the dimension in the $\nu$-direction. In the majority of our experiments, $\zeta_{\nu} =1$ or $2$, corresponding to the second-order approximation to the absorption boundary conditions on staggered and collocated grids, respectively.
Under the assumption that $\bar{\Lambda}_{\nu}$ is diagonalizable, we denote by:
\begin{itemize}
    \item $V_{\nu} = [V_{\nu,1}, \dots, V_{\nu,N_{\nu}}]$ the matrix of eigenvectors, where $V_{\nu,j}$ is the $j$-th eigenvector and $v_{\nu,i,j}$ represents the $i$-th component of $V_{\nu,j}$;
    \item $V_{\nu}^{-1}$ its inverse, where $V^{-1}_{\nu,j}$ is the $j$-th column and $v^{-1}_{\nu,i,j}$ represents the $i$-th component of $V^{-1}_{\nu,j}$; 
    \item $D_{\nu} = \operatorname{diag}(d_{\nu,1}, \dots, d_{\nu,N_{\nu}})$ the diagonal matrix of corresponding eigenvalues.
\end{itemize}

The proposed preconditioning matrix has the Kronecker product structure:
\begin{equation*}
A_p = I_z \otimes \left(\alpha_y\bar{\Lambda}_y\otimes I_x + \alpha_xI_y \otimes \bar{\Lambda}_x\right) + \left(\bar{\Lambda}_z + K_z - 2(\alpha_x+\alpha_y+1)I_z\right) \otimes I_{xy},
\end{equation*}
with the following components:
\begin{itemize}
    \item Grid aspect ratios: $\alpha_d = h_z^2/h_d^2$ for $d \in \{x,y\}$
    \item Depth-dependent coefficient matrix: $  K_z = h_z^2 \diag(k_0^2(z_1),  \dots, k_0^2(z_{N_z})) $
    \item Identity matrices: 
        \begin{itemize}
            \item $I_d$: $N_d \times N_d$ identity ($d \in \{x,y,z\}$)
            \item $I_{xy} = I_y \otimes I_x$: Horizontal subspace identity
        \end{itemize}
\end{itemize}

The preconditioning system $A_p U = Y$ can be expressed in block-tridiagonal form as:
\begin{align}
&(\gamma_{z}I_{xy}+B_{1})U_{1}+\zeta_{z}U_{2}=Y_{1}, \nonumber \\
&U_{l-1}+B_{l}U_{l}+U_{l+1}=Y_{l},\quad l=2,\ldots,N_{z}-1, \label{VecPre_system} \\
&\zeta_{z}U_{N_{z}-1}+(\gamma_{z}I_{xy}+B_{N_{z}})U_{N_{z}}=Y_{N_{z}}, \nonumber
\end{align}
where:
\begin{itemize}
    \item The block operators $B_l$ combine the horizontal components:
    \begin{equation*}
    B_l = \alpha_y \bar{\Lambda}_y \otimes I_x + \alpha_x I_y \otimes \bar{\Lambda}_x + \left(h_z^2 k_0^2(z_l) - 2(\alpha_x + \alpha_y + 1)\right) I_{xy}, \quad l = 1, \dots, N_z.
     \end{equation*}    \item The solution vector $U$ is organized by horizontal slices:
    \begin{equation*}
    U_l = \left(U_{1,1,l}, \dots, U_{N_x,1,l}, \dots, U_{1,N_y,l}, \dots, U_{N_x,N_y,l}\right)^T, \quad l = 1, \dots, N_z.
    \end{equation*}
    \item The right-hand side $Y \in \mathbf{C}^{N_x \times N_y \times N_z}$ is updated at each iteration.
\end{itemize}
The diagonalization of $B_l$ follows directly from the eigenvector transformations:
\begin{align*}
&(V_y^{-1} \otimes V_x^{-1})(\alpha_y \bar{\Lambda}_y \otimes I_x + \alpha_x I_y \otimes \bar{\Lambda}_x)(V_y \otimes V_x) \\
&= \alpha_y (V_y^{-1}\bar{\Lambda}_y V_y) \otimes (V_x^{-1}V_x) + \alpha_x (V_y^{-1}V_y) \otimes (V_x^{-1}\bar{\Lambda}_x V_x) \\
&= \alpha_y D_y \otimes I_x + \alpha_x I_y \otimes D_x.
\end{align*}
The preconditioned system takes its fully decoupled form after the eigenvector transformation:
\begin{align}
&(\gamma_{z}I_{xy}+\bar{B}_{1})\bar{U}_{1}+\zeta_{z}\bar{U}_{2}=\bar{Y}_{1}, \nonumber \\
&\bar{U}_{l-1}+\bar{B}_{l}\bar{U}_{l}+\bar{U}_{l+1}=\bar{Y}_{l},\quad l=2,\ldots,N_{z}-1, \\
&\zeta_{z}\bar{U}_{N_{z}-1}+(\gamma_{z}I_{xy}+\bar{B}_{N_{z}})\bar{U}_{N_{z}}=\bar{Y}_{N_{z}}, \nonumber
\end{align}
with the transformed variables and operators defined by:
     \begin{align} \label{DIAGprec}
    &\bar{B}_l = \underbrace{\alpha_y D_y \otimes I_x + \alpha_x I_y \otimes D_x}_{\text{Decoupled horizontal terms}} 
    + \underbrace{(h_z^2 k_0^2(z_l) - 2(\alpha_x + \alpha_y + 1))I_{xy}}_{\text{Vertical coupling term}}, \\
   & \bar{U}_l = (V_y^{-1} \otimes V_x^{-1}) U_l, \ \bar{Y}_l = (V_y^{-1} \otimes V_x^{-1}) Y_l, \ \ l = 1,\dots,N_z.  \nonumber
    \end{align}
 
The system further decomposes into $N_x \times N_y$ independent tridiagonal systems:
\begin{align}
&(\gamma_{z}+\bar{b}_{i,j,1})\bar{u}_{i,j,1}+\zeta_{z}\bar{u}_{i,j,2}=\bar{y}_{i,j,1}, \nonumber \\
&\bar{u}_{i,j,l-1}+\bar{b}_{i,j,l}\bar{u}_{i,j,l}+\bar{u}_{i,j,l+1}=\bar{y}_{i,j,l}, \quad 
\begin{cases}l=2,\ldots,N_{z}-1\\i=1,\ldots,N_{x}\\j=1,\ldots,N_{y}\end{cases} \label{Indep_precond_system} \\
&\zeta_{z}\bar{u}_{i,j,N_{z}-1}+(\gamma_{z}+\bar{b}_{i,j,N_{z}})\bar{u}_{i,j,N_{z}}=\bar{y}_{i,j,N_{z}}. \nonumber
\end{align}
where each diagonal entry combines the eigenvalues $d_{x,i}$ and $d_{y,j}$ with the vertical coupling term:
\begin{equation}
\bar{b}_{i,j,l} = d_{x,i} + d_{y,j} + h_z^2 k_0^2(z_l) - 2(\alpha_x + \alpha_y + 1).
\end{equation}
The solution of the system $A_pU=Y$ is computed via the EigT algorithm, which consists of four steps:
\begin{align}
&\text{1. \textbf{Precomputation:} Compute } D_{\nu}, V_{\nu}, V_{\nu}^{-1} \text{ for } \nu = x, y; \nonumber \\
&\text{2. \textbf{Forward Transform:} Compute } \bar{Y}_l = (V_y^{-1} \otimes V_x^{-1}) Y_l, \nonumber \\
&\qquad l=1,\dots,N_z; \nonumber \\
&\text{3. \textbf{Vertical Solve:} Solve (\ref{Indep_precond_system}) in parallel for all } (i,j); \label{EIG_SOLVER} \\
&\text{4. \textbf{Inverse Transform:} Compute } U_l = (V_y \otimes V_x) \bar{U}_l, \nonumber \\
&\qquad l=1,\dots,N_z. \nonumber
\end{align}
Step 1 in the solution can be precomputed and usually takes less than $5\%$ of total computational time required by Steps 2-4. If the number of iterations in the preconditioned iterative method is greater than ten,
this computational cost becomes negligible. 

\subsection{Partial FFT (PFFT) Preconditioning Approach}

The PFFT method utilizes a modified preconditioning matrix $A_p^F$ from our previous work \cite{hrt}:
\begin{equation}
A_p^F = I_z \otimes (\alpha_y \Lambda_y \otimes I_x + \alpha_x I_y \otimes \Lambda_x)  + (\bar{\Lambda}_z + K_z - 2(\alpha_x+\alpha_y+1)I_z) \otimes I_{xy}, \label{TFprec}
\end{equation}
where $\Lambda_{\nu}$ (for $\nu = x, y$) is the matrix $\Lambda$ defined in (\ref{lambda_mat_definition_new}) with size $N_{\nu} \times N_{\nu}$.

The solution of the system $A_p U = Y$ can be implemented efficiently using 
the following three-step PFFT algorithm, which uses the FFT-based solver $A_p^F$:
\begin{align}
&\text{1. \textbf{Initial FFT step:} Solve } A_p^F \Theta = Y \text{ using 2D FFT} \nonumber \\
&\quad \text{in } O(N_xN_yN_z \log (N_xN_y)) \text{ operations}. \nonumber \\
&\text{2. \textbf{Boundary Correction:} Use (\ref{VecPre_system}-\ref{Indep_precond_system}) to solve} \nonumber \\
&\quad \text{for } W = U - \Theta \text{ at boundary and near-boundary points only,} \nonumber \\
&\quad \text{requiring } O((N_x^2 + N_x N_y + N_y^2)N_z) \text{ operations:} \nonumber \\
&\quad A_p W = Y - A_p \Theta = (A^F_p - A_p)\Theta. \label{PFFT_SOLVER} \\
&\text{3. \textbf{FFT Finalization:} Solve } A^F_p U = Y + (A^F_p - A_p)(\Theta + W). \nonumber
\end{align}
\noindent
Note that in Step 2, we do not solve for the full vector $W$. Because the matrix $(A_p^F - A_p)$ is non-zero only at entries corresponding to the boundary and adjacent interior points (see its structure below in (\ref{PFFT_INV})), the right-hand side $(A_p^F - A_p)\Theta$ is sparse, and we only require---and thus only compute---the values of $W$ at those locations. The full solution $U$ is then obtained in Step 3 by applying the fast $O(N_xN_yN_z \log (N_xN_y))$ FFT-based solver $A_p^F$ to a corrected right-hand side. We also use the inverse transform matrix directly on the last stage of the first step to find the solution only at the boundary and near-boundary points (without inverse FFT) and utilize the transformed $Y$ from the first step in the forward FFT stage of the third step of the algorithm, along with direct forward transformation of the sparse vector $(A^F_p - A_p)(\Theta + W)$. This optimization does not affect the theoretical complexity estimate but reduces the number of FFT transforms of the 3D vector from four to two.

The $O((N_x^2 + N_x N_y + N_y^2)N_z)$ complexity of the second step arises from the sparse structure of the residual vectors $Y_l = (A_p^F - A_p)\Theta_l$ and $(A^F_p - A_p)W$. The vector $Y_l$ decomposes into $Y_l = \Theta_l^x + \Theta_l^y$, where

\begin{align*}
&\Theta^x_l  = \begin{cases}
-\gamma_x \Theta_{i,j,l}, & \text{at } i=1,N_x,\ j=1, \dots , N_y,\ l=1, \dots , N_z, \\
(1-\zeta_x) \Theta_{i,j,l}, & \text{at } i=2,N_x-1,\ j=1, \dots , N_y,\ l=1, \dots , N_z, \\
0, & \text{otherwise},
\end{cases}  \\
&\Theta^y_l  = \begin{cases}
-\gamma_y \Theta_{i,j,l}, & \text{at } j=1,N_y,\ i=1, \dots , N_x,\ l=1, \dots , N_z, \\
(1-\zeta_y) \Theta_{i,j,l}, & \text{at } j=2,N_y-1,\ i=1, \dots , N_x,\ l=1, \dots , N_z, \\
0, & \text{otherwise}.
\end{cases}
\end{align*}
Then the transformation of the right-hand side $\bar{Y}_l = (V_y^{-1} \otimes V_x^{-1}) Y_l$ 
(see Step 2 in the EigT procedure, Eq.~\ref{EIG_SOLVER}) can be computed as
 \begin{align}
&\text{Step 1: } S_{x,m} = V_x^{-1} Y_{m,l}, \quad m=1,2,N_y-1,N_y;  \nonumber\\
&\text{Step 2: } \bar{Y}_{j,l} = \sum_{m=1,2,N_y-1,N_y} v^{-1}_{y,j,m} S_{x,m}  +  \sum_{\nu=1,2,N_x-1,N_x} V_{x,\nu}^{-1} \sum_{m=3}^{N_y-2} (v^{-1}_{y,j,m} Y_{\nu,m,l}), \nonumber \\
&\quad\quad\quad\quad j=1, \dots,N_y,\ l=1, \dots,N_z.  \nonumber
\end{align}
The transformation of $Y_l$ can be calculated independently for every $l=1, \dots,N_z$ and requires $O(N_x^2 + N_x N_y)$ operations for each $l$. We note that if $\zeta_{\nu}=1$ ($\nu=x,y$), which corresponds to the two-point approximation of the absorption boundary condition on the staggered grids, then the second step becomes
 \begin{align}
&\text{Step 2*: } \bar{Y}_{j,l} = \sum_{m=1,N_y} v^{-1}_{y,j,m} S_{x,m}  +  \sum_{\nu=1,N_x} V_{x,\nu}^{-1} \sum_{m=2}^{N_y-1} (v^{-1}_{y,j,m} Y_{\nu,m,l}), \nonumber \\
&\quad\quad\quad\quad j=1, \dots,N_y,\ l=1, \dots,N_z.  \nonumber
\end{align}
Similarly, we can develop the formula for the inverse transform step in (\ref{EIG_SOLVER}). The key now is the sparsity of the vector $(A^F_p - A_p)W$ in the third step of (\ref{PFFT_SOLVER}), i.e., we only need to find $W$ at the boundary points. This radically reduces the number of necessary 1D eigenvector transformations in the implementation of $W_l = (V_y \otimes V_x) \bar{W}_l$ for $l=1, \dots,N_z$. This inverse transformation can be computed in three steps::
 \begin{align}
&\text{Step 1: } W_{j,l} = V_x \sum_{m=1}^{N_y} (v_{y,j,m} \bar{W}_{m,l}), \quad j=1, N_y;  \nonumber\\
&\text{Step 2: } s_{i,j} = \sum_{m=1}^{N_x} v_{x,i,m} \bar{W}_{m,j,l}, \quad i=1, N_x;\ j=1, \dots,N_y;  \label{PFFT_INV} \\
&\text{Step 3: } W_{i,j,l} = \sum_{m=1}^{N_y} (s_{i,m} v_{y,j,m}), \quad i=1, N_x;\ j=2, \dots,N_y-1,  \nonumber \\
&\quad\quad\quad\quad l=1, \dots,N_z.  \nonumber
\end{align}
The inverse transformation of $\bar{W}$ also requires $O(N_x N_y + N_y^2)$ operations at each horizontal level $l=1, \dots,N_z$. We also note that the values of $\Theta$ are required only at the boundary. Therefore, we use the procedure described in (\ref{PFFT_INV}) to calculate the required values of $\Theta$ in the first step of the PFFT procedure (\ref{PFFT_SOLVER}) using the analytic eigenvector formulas given in (\ref{LMB_egv}).   

In this paper, we construct an efficient iterative solver based on preconditioned GMRES \cite{s}, leveraging the preconditioners developed earlier. While general convergence results for GMRES can be found in \cite{s}, our specific convergence estimates for the lower-order preconditioner in high-order approximation schemes are derived for model cases in Theorem \ref{last_1D_theorem} and extended in our prior work \cite{yg2}.

The key to our iterative method's rapid convergence lies in a modified eigenvector transform approach for solving the preconditioning system. This modification enables the inclusion of low-accuracy approximations of absorption boundary conditions in the preconditioning matrix $A_p$ while maintaining computational efficiency.

The proposed algorithms are ``embarrassingly parallel'' as each horizontal slice can be processed independently. Their implementation in OpenMP, MPI, or hybrid environments follows the same paradigm as the FFT-GMRES combination methods detailed in \cite{ggl}. In the following section, we present numerical experiments demonstrating the efficacy of the developed preconditioners for general subsurface scattering problems.

\section{Numerical Results}

This section presents numerical experiments demonstrating the performance of the proposed EigT- and PFFT-preconditioned GMRES solvers. All algorithms were implemented in MATLAB and executed on a machine equipped with an Apple M3 Max processor (16-core, 4.05 GHz max) and 128 GB of RAM. To ensure a fair comparison of solver performance, all computations were restricted to a single CPU core, disabling MATLAB's inherent multithreading capabilities. The right-preconditioned restarted GMRES(20) method, available via MATLAB's built-in \texttt{gmres} function, was used as the iterative solver.

\subsection{3D Test Problem}

We first investigate the performance of the developed preconditioners using a 3D test problem similar to the 2D test presented in \cite{gkl}. The problem consists of solving the 3D Helmholtz equation with a constant wavenumber $k_0$ and a right-hand side $f$ constructed to yield a known analytic solution $u$ that satisfies the absorption boundary conditions.

The source function $f$ was selected such that the true solution is given by the separable function:
\[
u(x,y,z) = \phi(x)\phi(y)\phi(z),
\]
with $(x,y,z) \in \Omega = [0,1] \times [0,1] \times [0,1]$, and where
\[
\phi(x)=\exp(ik_{0}(x-a))+\exp(-ik_{0}(x-b))-2,
\]
with $a=0$, $b=1$, and $k_{0}^{2}=439.2$. This wavenumber corresponds to the propagation of an electromagnetic wave in non-attenuating media at a frequency of 1 GHz.

In our experiments, we use the following error and residual measures for the initial approximation $U^0$, the computed numerical solution $U$, and the analytic solution $u$:
\begin{itemize}
    \item \textbf{rel-res} (relative residual): $\|AU - F\|_2 / \|AU^0 - F\|_2$,
    \item \textbf{relmax-res} (relative maximum residual): $\|AU - F\|_{\infty} / \|AU^0 - F\|_{\infty}$,
    \item \textbf{$L_2$-err} (relative $L_2$ error): $\|u - U\|_2 / \|u\|_2$,
    \item \textbf{max-err} (maximum absolute error): $\|u - U\|_{\infty}$,
    \item \textbf{relmax-err} (relative maximum error): $\|u - U\|_{\infty} / \|u\|_{\infty}$.
\end{itemize}

The iterative GMRES process was terminated when the relative residual (\textbf{rel-res}) reached $10^{-10}$. In the results presented below, $N_0$ denotes the iteration count at which the 
convergence criterion was first satisfied, and $T_P$ denotes the corresponding 
CPU time in seconds.

\subsubsection{Transformation Methods Comparison}

In the first two tests, we compare the direct transformation approach (EigT) from (\ref{EIG_SOLVER}) and the PFFT method from (\ref{PFFT_SOLVER}) for solving the second-order central difference approximation of the 3D Helmholtz equation with a constant wavenumber $k_0^2 = 439.2$. The primary goal is to evaluate the computational efficiency of these two direct solvers, which are subsequently used as preconditioners.

We measure the computational time for the key components of each method:
\begin{itemize}
    \item The 2D sine transform (applied to all $N_z$ slices), used in Steps 1 and 3 of the PFFT method and implemented using the standard FFTW package.
    \item The full eigenvector transformation, used in Steps 2 and 4 of the EigT method (\ref{EIG_SOLVER}).
    \item The boundary correction steps (Steps 1 and 2 of the PFFT method, \ref{PFFT_SOLVER}), which constitute the first transformation step of the PFFT algorithm.
\end{itemize}
All transforms are applied to the right-hand side vector of the test problem on a sequence of increasingly finer grids. The eigenvector transform has an analytical complexity of $O(N_x N_y N_z (N_x + N_y))$ floating-point operations (FLOPs), while the FFT-based approach has an asymptotic complexity of $O(N_x N_y N_z (\log N_x + \log N_y))$ FLOPs for transforming $N_z$ 2D horizontal slices.

Table~\ref{tab:transform_times} presents the timing results. The first column lists the grid size. The second column shows the time required to solve the two auxiliary 1D eigenvalue problems (a one-time setup cost for the EigT method). The last three columns display the total computational time for the sine transform, the full EigT transform, and the PFFT transform (using one layer of boundary points), respectively.

\begin{table}[h!]
\centering
\begin{tabular}{|c|c|c|c|c|}
\hline
\textbf{Grid} & \textbf{1D Eig.Solve(s)} & \textbf{Sine Transform(s)} & \textbf{EigT(s)} & \textbf{PFFT(s)} \\ \hline
$100^3$ & 0.02 & 0.077 & 0.05 & 0.14   \\ \hline
$200^3$ & 0.085 & 0.83 & 0.59 & 1.12  \\ \hline
$400^3$ & 0.59 & 5.3 & 9.2 & 7.7  \\ \hline
$600^3$ & 1.53 & 18.2 & 45.5 & 28.2  \\ \hline
$800^3$ & 3.03 & 64.8 & 144.9 & 100.1  \\ \hline
$127^3$ & 0.03 & 0.064 & 0.11 & 0.16  \\ \hline
$255^3$ & 0.16 & 0.55 & 1.58 & 1.11  \\ \hline
$511^3$ & 1.03 & 4.68 & 24.2 & 11.8  \\ \hline
$767^3$ & 2.75 & 18.4 & 120.5 & 44.9  \\ \hline
\end{tabular}
\caption{Comparison of computational time (in seconds) for different transformation methods.}
\label{tab:transform_times}
\end{table}

Our experiments reveal two distinct performance regimes:
\begin{itemize}
    \item For smaller grids ($N_x, N_y < 255$) or those with grid sizes that are non-optimal for FFTs (e.g., containing large prime factors), the direct eigenvector transformation (EigT) is more efficient than the PFFT approach. This is due to the higher constant factors and precomputation overhead associated with the FFTW library.
    \item For larger grids ($N_x, N_y > 256$) with sizes that are powers of two or highly composite (containing only small prime factors), the PFFT method becomes the superior choice. Its theoretical $O(N \log N)$ asymptotic complexity advantage is realized, clearly outperforming the $O(N^2)$ scaling of the EigT method.
\end{itemize}
The transition between these regimes occurs near the $255^3$ to $400^3$ grid size, where the FFT's algorithmic efficiency begins to dominate the computational cost.

\subsubsection{Direct Solution by the Low-Order Preconditioners}

Table~\ref{tab:direct_solve} compares the performance of the EigT and PFFT methods for directly solving the second-order finite difference approximation of the 3D Helmholtz equation. Results are shown for two boundary condition approximations: two-point on staggered grids (\textbf{EigT2}/\textbf{PFFT2}) and three-point on collocated grids (\textbf{EigT3}/\textbf{PFFT3}). All computational times are reported in seconds.

\begin{table}[h!]
\centering
\small 
\begin{tabular}{|c|c|c|c|c|c|c|}
\hline
\textbf{Grid} & \textbf{max-err} & \textbf{$L_2$-err} & \textbf{EigT2} & \textbf{PFFT2} & \textbf{EigT3} & \textbf{PFFT3} \\ \hline
$100^3$ & $7.74 \times 10^{-1}$ & $1.60 \times 10^{-2}$ & 0.10 & 0.27 & 0.11 & 0.34  \\ \hline
$200^3$ & $1.93 \times 10^{-1}$ & $4.00 \times 10^{-3}$ & 1.38 & 2.60 & 1.63 & 2.86 \\ \hline
$400^3$ & $4.80 \times 10^{-2}$ & $9.90 \times 10^{-4}$ & 20.41 & 19.79 & 24.63 & 25.73\\ \hline
$600^3$ & $2.10 \times 10^{-2}$ & $4.40 \times 10^{-4}$ & 95.08 & 61.78 & 118.00 & 86.26\\ \hline
$800^3$ & $1.20 \times 10^{-2}$ & $2.50 \times 10^{-4}$ & 293.29 & 234.06 & 367.00 & 290.00\\ \hline
$127^3$ & $4.79 \times 10^{-1}$ & $1.00 \times 10^{-2}$ & 0.24 & 0.41 & 0.29 & 0.40\\ \hline
$255^3$ & $1.19 \times 10^{-1}$ & $2.00 \times 10^{-3}$ & 3.35 & 2.83 & 4.12 & 3.70\\ \hline
$511^3$ & $3.00 \times 10^{-2}$ & $6.10 \times 10^{-4}$ & 50.40 & 22.90 & 61.80 & 35.80 \\ \hline
$767^3$ & $1.30 \times 10^{-2}$ & $2.70 \times 10^{-4}$ & 247.80 & 109.80 & 306.50 & 172.00 \\ \hline
\end{tabular}
\caption{Comparison of the direct solution time (seconds) and solution errors for the second-order preconditioning system.}
\label{tab:direct_solve}
\end{table}

For reference, the standard \textsc{Matlab} backslash operator required 41.8~s, 148.8~s, and 448.7~s for $50^3$, $60^3$, and $70^3$ grids, respectively, while failing to solve a $100^3$ system within 30 minutes. In stark contrast, both of our transformation methods solve these larger systems in fractions of a second.

The results confirm the following:
\begin{itemize}
    \item The EigT method is more efficient for smaller grids or those with sizes non-optimal for FFTs (typically below $255^3$).
    \item The PFFT method outperforms EigT for larger grids (typically $\geq 255^3$) with sizes that are optimal for FFTs, realizing its theoretical $O(N \log N)$ complexity advantage.
\end{itemize}

Columns 2--3 report the maximum absolute error  \textbf{max-err} and the relative $L_2$ error (\textbf{$L_2$-err}) of the finite-difference solution. The consistent reduction in error as the grid is refined confirms that our methods maintain the second-order accuracy of the underlying scheme. We observe that even on an $800^3$ grid, the maximum error remains on the order of $10^{-2}$.

In the next section, we consider the implementation of high-order compact schemes using the developed EigT and PFFT preconditioners within an iterative solver.

\subsubsection{Preconditioned GMRES Solution for Higher-Order Schemes}

This section presents numerical experiments for solving the fourth- and sixth-order compact schemes (\ref{scheme4}, \ref{scheme6}) on collocated grids, preconditioned with the proposed EigT3 and PFFT3 methods. For comparison, we also solve the same high-order problems using preconditioners based on second-, fourth-, and sixth-order schemes with zero Dirichlet boundary conditions, as analyzed in \cite{ggl}. In the results presented below, $N_0$ denotes the iteration count at which the convergence criterion was first satisfied, and $T_P$ denotes the corresponding total processor time in seconds. These reference preconditioners are denoted FFT2, FFT4, and FFT6. The results are compared in Tables~\ref{tab:fourth_order_results} and~\ref{tab:sixth_order_results}.

\begin{table}[h!]
\centering
\begin{tabular}{|c|c|c|c|c|c|c|c|c|c|c|}
\hline
\textbf{Grid} & \textbf{max-err} & \textbf{$L_2$-err} & \multicolumn{2}{c|}{\textbf{EigT3}} & \multicolumn{2}{c|}{\textbf{PFFT3}} & \multicolumn{2}{c|}{\textbf{FFT2}} & \multicolumn{2}{c|}{\textbf{FFT4}} \\
\cline{4-11}
 & & & $\mathbf{N_0}$ & \textbf{TP} & $\mathbf{N_0}$ & \textbf{TP} & $\mathbf{N_0}$ & \textbf{TP} & $\mathbf{N_0}$ & \textbf{TP} \\
\hline
 $50^3$ & 0.02 & 5.25e-04 & 8 & 0.31 & 8 & 0.48 & 46 & 2.15 & 37 & 1.41 \\
\hline
$100^3$ & 1.25e-03 & 3.12e-05 & 5 & 1.82 & 5 & 2.81 & 45 & 19.3 & 42 & 15.3 \\
\hline
$200^3$ & 7.68e-05 & 1.91e-06 & 4 & 17.6 & 4 & 24.4 & 57 & 246 & 54 & 234 \\
\hline
$400^3$ & 4.75e-06 & 1.18e-07 & 3 & 155 & 3 & 144 & 75 & 2588 & 72 & 2538 \\
\hline
$64^3$ & 7.64e-03 & 1.91e-04 & 7 & 0.85 & 7 & 0.60 & 43 & 4.18 & 37 & 2.84 \\
\hline
$127^3$ & 4.78e-04 & 1.19e-05 & 5 & 4.45 & 5 & 5.72 & 50 & 40.9 & 48 & 39.7 \\
\hline
$255^3$ & 2.89e-05 & 7.19e-07 & 4 & 34.2 & 4 & 34.1 & 60 & 515 & 58 & 494 \\
\hline
$511^3$ & 1.78e-06 & 4.42e-08 & 3 & 386 & 3 & 249 & 87 & 4970 & 86 & 5650 \\
\hline
\end{tabular}
\caption{Comparison of the iterative GMRES solution for the fourth-order preconditioning system.}
\label{tab:fourth_order_results}
\end{table}

\begin{table}[h!]
\centering
\begin{tabular}{|c|c|c|c|c|c|c|c|c|c|c|}
\hline
\textbf{Grid} & \textbf{max-err} & \textbf{$L_2$-err} & \multicolumn{2}{c|}{\textbf{EigT3}} & \multicolumn{2}{c|}{\textbf{PFFT3}} & \multicolumn{2}{c|}{\textbf{FFT2}} & \multicolumn{2}{c|}{\textbf{FFT6}} \\
\cline{4-11}
 & & & $\mathbf{N_0}$ & \textbf{TP} & $\mathbf{N_0}$ & \textbf{TP} & $\mathbf{N_0}$ & \textbf{TP} & $\mathbf{N_0}$ & \textbf{TP} \\
\hline
 $50^3$ & 2.57e-04 & 7.06e-06 & 8 & 0.26 & 8 & 0.44 & 43 & 1.89 & 36 & 1.52 \\
\hline
$100^3$ & 3.71e-06 & 9.97e-08 & 5 & 2.06 & 5 & 3.13 & 45 & 21.6 & 42 & 20.7 \\
\hline
$200^3$ & 5.63e-08 & 1.49e-09 & 4 & 19.3 & 4 & 25.6 & 56 & 261 & 55 & 256 \\
\hline
$400^3$ & 1.24e-07 & 7.51e-11 & 3 & 168 & 3 & 168 & 75 & 2883 & 71 & 2562 \\
\hline
$63^3$ & 6.27e-05 & 1.68e-06 & 7 & 0.47 & 7 & 0.68 & 40 & 3.13 & 37 & 2.86 \\
\hline
$127^3$ & 8.87e-07 & 2.33e-08 & 5 & 4.40 & 5 & 5.96 & 50 & 40.8 & 48 & 40.27 \\
\hline
$255^3$ & 5.24e-08 & 3.44e-10 & 4 & 38.7 & 4 & 37.2 & 60 & 407.9 & 59 & 395.2 \\
\hline
$511^3$ & 9.81e-07 & 1.83e-11 & 3 & 462 & 3 & 318 & 87 & 6314 & 86 & 6184 \\
\hline
\end{tabular}
\caption{Comparison of the iterative GMRES solution for the sixth-order preconditioning system.}
\label{tab:sixth_order_results}
\end{table}

The results demonstrate several key findings. The convergence of the approximate solutions exhibits the expected fourth- and sixth-order accuracy, although for very fine grids (smaller than $1/200$) the maximum error for the sixth-order scheme plateaus due to the limitations of double-precision arithmetic.

The proposed GMRES-EigT and GMRES-PFFT combinations significantly outperform the same iterative method preconditioned with the FFT-based approaches (FFT2, FFT4, FFT6). We observe that merely increasing the approximation order of the preconditioner while maintaining Dirichlet boundary conditions only marginally improves convergence. In contrast, incorporating a lower-order approximation of the radiation boundary condition within the preconditioner radically reduces the number of iterations required for convergence and, consequently, reduces the total computation time by more than an order of magnitude.

A particularly outstanding property of this approach is that the number of iterations to convergence decreases as the size of the linear system increases. This property, first illustrated in our 1D model problem, is clearly maintained in this 3D setting. For instance, solving the fourth- and sixth-order systems on a $511^3$ grid requires approximately 5 minutes and only 3 iterations with the PFFT preconditioner, whereas our previous FFT-GMRES algorithm required about 1.7 hours and 86 iterations to converge.


\begin{table}[h!]
\centering
\begin{tabular}{|c|c|c|c|c|c|c|c|c|}
\hline
\textbf{Grid} & \multicolumn{2}{c|}{\textbf{$k=10$}} & \multicolumn{2}{c|}{\textbf{$k=20$}} & \multicolumn{2}{c|}{\textbf{$k=50$}} & \multicolumn{2}{c|}{\textbf{$k=100$}} \\
\cline{2-9}
 & $\mathbf{N_0}$ & \textbf{TP} & $\mathbf{N_0}$ & \textbf{TP} & $\mathbf{N_0}$ & \textbf{TP} & $\mathbf{N_0}$ & \textbf{TP} \\
\hline
 $100^3$ & 4 & 2.35 & 5 & 2.69 & 9 & 5.49 & 34 & 19.6 \\
\hline
 $200^3$ & 3 & 20.9 & 4 & 25.0 & 6 & 34.4 & 12 & 62.0 \\
\hline
 $400^3$ & 3 & 171 & 3 & 175 & 5 & 254 & 7 & 340 \\
\hline
 $511^3$ & 3 & 312 & 3 & 312 & 4 & 384 & 6 & 527 \\
\hline
\end{tabular}
\caption{PFFT-GMRES convergence for the test problem with different wavenumbers $k$.}
\label{tab:diff_k_results}
\end{table}

We observe that the number of iterations until convergence is a non-increasing function of the grid size for all considered values of $k$, which is consistent with our 1D model estimates. These results confirm the high efficiency of implementing the compact fourth- and sixth-order approximation schemes using the proposed PFFT-GMRES-type algorithm for a wide range of wavenumbers.

\subsubsection{Scattering Problem with a Single Inclusion}

Next, we will apply the developed method to a scattering problem. In this series of tests, we compute the wavefield response to a single spherical inclusion in non-attenuating media. We consider the same rectangular computational domain $\Omega = \{ 0 \leq x,y,z \leq 1 \}$ as in previous tests. The target inclusion is embedded within the subdomain $\Omega \cap \{ z > 0.5 \}$. The wavenumber function $k^2(x,y,z)$ is defined as:
\begin{equation}
k^2(\mathbf{x}) =\left\{ 
\begin{array}{ll}
1050 + 2.26i, & \text{for } (x-x_1)^2 + (y-y_1)^2 + (z-z_1)^2 \leq r_1^2 \\ 
439.2, & \text{otherwise}
\end{array}
\right.
\label{Inc_cf} 
\end{equation}
where $(x_1, y_1, z_1) = (0.5, 0.5, 0.7)$ and $r_1 = 0.1$. The right-hand side in (\ref{problem}) is given by $f = -(k^2 - k^2_0)u_0$, where $k^2_0 = 439.2$ and $u_0 = e^{ik_0z}$.

For these experiments, we use staggered grids with grid point coordinates defined by $\nu = h_{\nu}/2 + h_{\nu} \cdot j$, where $h_{\nu} = 1/N_{\nu}$, $j = 0, \dots, N_{\nu} - 1$, and $\nu = x, y, z$. We employ a second-order central difference approximation of the Sommerfeld-like boundary conditions (\ref{Sombc}). The PFFT method only requires updating the solution at one boundary layer, allowing us to use the most efficient PFFT2 algorithm as a preconditioner.

Table~\ref{tab:scattering_results} shows the convergence of the sixth-order preconditioned GMRES method with FFT6, EIG2, and PFFT2 preconditioners, using the same stopping criterion as previous tests. The relative residual (\textbf{rel-res}) is reported only for the PFFT-GMRES algorithm. Second- and fourth-order approximation methods display similar convergence properties.

\begin{table}[h!]
\centering
\begin{tabular}{|c|c|c|c|c|c|c|c|}
\hline
\textbf{Grid} & \textbf{rel-res} & \multicolumn{2}{c|}{\textbf{EigT2}} & \multicolumn{2}{c|}{\textbf{PFFT2}} & \multicolumn{2}{c|}{\textbf{FFT6}} \\
\cline{3-8}
 & & $\mathbf{N_0}$ & \textbf{TP} & $\mathbf{N_0}$ & \textbf{TP} & $\mathbf{N_0}$ & \textbf{TP} \\
\hline
$100^3$ & 3.07e-09 & 12 & 4.59 & 12 & 7.13 & 68 & 34.1 \\
\hline
$200^3$ & 9.13e-09 & 11 & 41 & 11 & 55 & 84 & 366 \\
\hline
$400^3$ & 9.37e-09 & 11 & 1099 & 11 & 775 & 100 & 5858 \\
\hline
$127^3$ & 4.36e-09 & 12 & 10.8 & 12 & 11.7 & 72 & 58.9 \\
\hline
$255^3$ & 1.1e-08 & 11 & 93 & 11 & 89 & 93 & 637 \\
\hline
$511^3$ & 9.09e-09 & 11 & 1101 & 11 & 771 & 100 & 5860 \\
\hline
\end{tabular}
\caption{Comparison of sixth-order preconditioned GMRES solution for non-homogeneous coefficient case.}
\label{tab:scattering_results}
\end{table}

The developed algorithms maintain the same favorable convergence properties for non-constant coefficients as observed in homogeneous media: the number of iterations until convergence remains a non-increasing function of grid size. The high-order FFT preconditioner (FFT6) reached the preset maximum of 100 iterations for grid sizes $400^3$ and $511^3$ without achieving convergence. The PFFT algorithm demonstrates the minimum total computational time required to reach the stopping criterion on all considered grid sizes larger than $200^3$.

To further illustrate the properties of the developed PFFT-GMRES algorithm, we examine the convergence history on a $511^3$ grid and the distribution of the real part of the numerical solution throughout the computational domain.

\begin{figure}[h!]
\centering
\includegraphics[height=8cm]{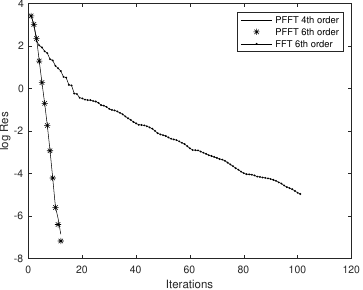}
\caption{Convergence history of the preconditioned PFFT-  and FFT-GMRES methods.}
\label{fig:convergence_history}
\end{figure}

Figure~\ref{fig:convergence_history} shows the convergence history of the relative residual (\textbf{rel-res}) for the preconditioned PFFT- and FFT-GMRES methods on a $511^3$ grid, using a logarithmic scale to clearly display the convergence behavior.

Figure~\ref{fig:field_amplitude} shows grayscale plots of the real part of the wavefield amplitude $u$ throughout the computational domain on a $511^3$ grid. The figure displays slices through the domain at $x = 0.5$ and $y = 0.5$, revealing the scattering pattern induced by the inclusion.

\begin{figure}[h!]
\centering
\includegraphics[height=8cm]{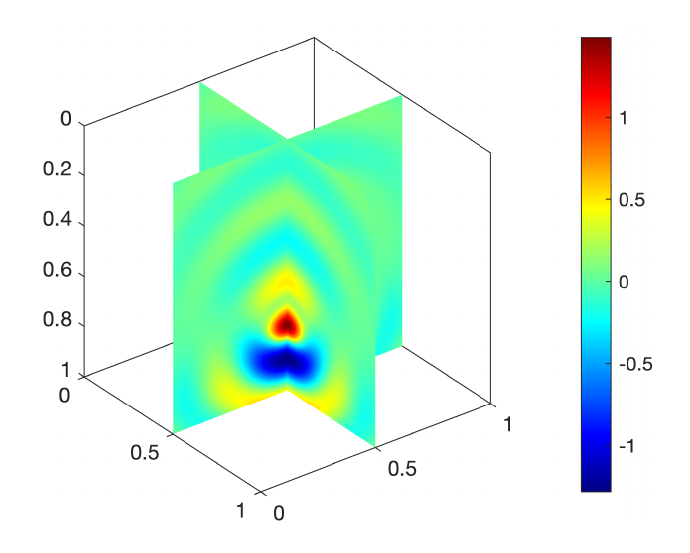}
\caption{Color plot of the real part of the wavefield amplitude $U$.}
\label{fig:field_amplitude}
\end{figure}

\section{Conclusions}

This paper has developed, analyzed, and implemented novel 3D compact higher-order approximation methods based on Partial Fast Fourier Transform (PFFT) and Eigenvector Transform (EigT) preconditioned GMRES algorithms.

The core strategy combines higher-order approximation schemes with a lower-order preconditioner that incorporates accurate approximations of the physical boundary conditions, particularly the absorption boundary conditions essential for scattering problems. Analysis of typical test problems reveals the remarkable properties of the developed methods, such as the number of iterations to convergence decreasing as the number of grid points (and thus the size of the resulting matrix) increases. This approach is especially attractive in situations where a robust lower-order solver already exists, but the original boundary value problem demands a more accurate numerical approximation.

To analyze the properties of the developed PFFT and EigT preconditioners, we considered their performance as direct second-order solvers for the 3D Helmholtz equation with constant coefficients. This analysis considered both two-point (staggered grid) and three-point (collocated grid) implementations of the boundary conditions. The computational cost was shown to correspond to a near-optimal \(O(N \log N)\) complexity for the PFFT case, typical of FFT-type solvers. For smaller grids or those with sizes non-optimal for FFTs, the EigT preconditioner presents a more efficient alternative, despite its higher \(O(N^{4/3})\) computational complexity. These analytical results were confirmed numerically on test problems with known analytical solutions.

For the sixth-order compact approximation of the nonhomogeneous 3D Helmholtz equation with Sommerfeld radiation boundary conditions on a $512^3$ grid, the typical number of iterations to achieve a $10^{-10}$ reduction in the relative $L_2$-norm of the residual was only 11. The total computational time required was approximately 10 minutes on a single core of an Apple M3 Max processor. The method was successfully validated across realistic parameter ranges typical for electromagnetic scattering applications.

It is important to note that the Sommerfeld-like boundary condition implemented here \textbf{in the non-constant coefficient case} corresponds to a first-order approximation of the true Sommerfeld condition at infinity on an unbounded domain. Therefore, the direct application of higher-order approximations solely to the differential operator, without a corresponding high-order treatment of the boundary condition, is not always justified for scattering problems with varying material properties. For constant-coefficient validation problems, the boundary condition accuracy is less critical, and we use higher-order approximation schemes primarily to demonstrate convergence and accuracy. Future work will focus on extending this framework to more advanced boundary treatments, such as higher-order absorbing boundary conditions (ABCs) or Perfectly Matched Layers (PMLs) (see, e.g., [10]), and on large-scale parallel implementation using hybrid combinations of OpenMP, MPI, and CUDA programming models.

 \section*{Acknowledgments}

This research was supported in part by the U.S. Department of Energy, Office of Science, Office of Advanced Scientific Computing Research's Applied Mathematics program under Contract No. DE-AC02-05CH11231 at Lawrence Berkeley National Laboratory.

\section*{Declaration of generative AI and AI-assisted technologies in the manuscript preparation process}

During the preparation of this work, the authors used Grammarly.com and DeepSeek.com services to improve the readability and grammar of the manuscript. After using these services, the authors reviewed and edited the content as needed and take full responsibility for the content of the published article.












\bibliographystyle{elsarticle-num}

\begin{thebibliography}{99}

\bibitem{bs}
I.~Babuska and S.~Sauter.
Is the Pollution effect of the FEM avoidable for the Helmholtz equation considering high wave numbers?
\emph{SIAM Rev.}, 42:451--484, 2000.

\bibitem{bgt1}
A.~Bayliss, C.~Goldstein, and E.~Turkel.
An iterative method for the Helmholtz equation.
\emph{J. Comput. Phys.}, 49:443--467, 1983.

\bibitem{bgt2}
A.~Bayliss, C.~Goldstein, and E.~Turkel.
On accuracy conditions for the numerical computation of waves.
\emph{J. Comput. Phys.}, 59:396--404, 1985.

\bibitem{bpx}
J.H.~Bramble, J.E.~Pasciak, and J.~Xu.
The analysis of multigrid algorithms for nonsymmetric and indefinite problems.
\emph{Math. Comp.}, 51:389--414, 1988.

\bibitem{btt}
S.~Britt, S.~Tsynkov, and E.~Turkel.
Numerical simulation of time-harmonic waves in inhomogeneous media using compact high order schemes.
\emph{Commun. Comput. Phys.}, 9:481--496, 2011.

\bibitem{CheneyIsaacson}
M.~Cheney and D.~Isaacson.
Inverse problems for a perturbed dissipative half-space.
\emph{Inverse Problems}, 11:865--888, 1995.

\bibitem{ColtonKress}
D.~Colton and R.~Kress.
\emph{Inverse Acoustic and Electromagnetic Scattering Theory}.
Springer, 2nd edition, 1998.

\bibitem{KressNA}
R.~Kress.
\emph{Numerical Analysis}.
Springer, 1st edition, 1998.

\bibitem{dhr}
J.~Douglas Jr, J.L.~Hensley, and J.E.~Roberts.
Alternating-direction iteration method for Helmholtz problems.
Tech. Report No. 214, Mathematics Department, Purdue University, West Lafayette, IN, 1993.

\bibitem{eo}
H.C.~Elman and D.P.~O'Leary.
Efficient iterative solution of the three-dimensional Helmholtz equation.
\emph{J. Comput. Phys.}, 142:163--181, 1998.

\bibitem{eo1}
H.C.~Elman and D.P.~O'Leary.
Eigenanalysis of some preconditioned Helmholtz problems.
\emph{Numer. Math.}, 83:231--257, 1999.

\bibitem{EG}
O.~Ernst and G.H.~Golub.
A domain decomposition approach to solving the Helmholtz equation with a radiation boundary condition.
In A.~Quarteroni, H.~Periaux, Y.~Kuznetsov, and O.~Widdlund, editors, \emph{Domain Decomposition in Science and Engineering}, pages 177--192. Amer. Math. Soc., 1994.

\bibitem{GV}
G.H.~Golub and C.F.~VanLoan.
\emph{Matrix Computations}.
The Johns Hopkins University Press, Baltimore, MD, 2nd edition, 1989.

\bibitem{gy}
Y.A.~Gryazin.
Preconditioned Krylov subspace methods for sixth order compact approximations of the Helmholtz equation.
\emph{ArXiv:1212.1397 [math.NA]}, 2012.

\bibitem{yg1}
Y.A.~Gryazin.
A compact sixth order scheme combined with GMRES method for the 3D Helmholtz equation.
In \emph{Proc. of The 10th International Conference on Mathematical and Numerical Aspects of Waves}, pages 339--442, Vancouver, Canada, July 24--29 2011.

\bibitem{yg2}
Y.~A.~Gryazin.
Preconditioned Krylov subspace methods for sixth order compact approximations of the Helmholtz equation.
\emph{ISRN Computational Mathematics}, pages 1--15, 2014.
DOI: 10.1155/2014/745849.

\bibitem{yg3}
Y.~A.~Gryazin.
High order approximation compact schemes for forward subsurface scattering problems.
In \emph{Proceedings of the SPIE 9077, Radar Sensor Technology XVIII Conference}, pages 1--9, May 2014.
DOI: 10.1117/12.2050189.

\bibitem{ggl}
R.~Gonzales, Y.~Gryazin, and Y.T.~Lee.
Parallel FFT algorithms for high-order approximations on three-dimensional compact stencils.
\emph{Parallel Computing}, 103:102757, 2021.

\bibitem{gkl}
Y.A.~Gryazin, M.V.~Klibanov, and T.R.~Lucas.
GMRES computation of high frequency electrical field propagation in land mine detection.
\emph{J. Comput. Phys.}, 158:98--115, 2000.

\bibitem{Hayt}
W.H.~Hayt.
\emph{Engineering Electromagnetics}.
McGraw Hill, 3rd edition, 1974.

\bibitem{hrt}
E.~Heikkola, T.~Rossi, and J.~Toivanen.
Fast direct solution of the Helmholtz equation with a perfectly matched layer or an absorbing boundary condition.
\emph{Internat. J. Numer. Methods Engrg.}, 57(14):2007--2025, 2003.

\bibitem{hmmo}
J.~Heys, T.~Manteuffel, S.~McCormick, and L.~Olson.
Algebraic multigrid for higher-order finite elements.
\emph{J. Comp. Phys.}, 204:520--532, 2005.

\bibitem{l}
S.H.~Lui.
\emph{Numerical Analysis of Partial Differential Equations}.
Wiley, 2011.

\bibitem{kim}
S.~Kim.
Parallel multidomain iterative algorithms for the Helmholtz wave equation.
\emph{Appl. Numer. Math.}, 59:411--429, 1995.

\bibitem{nsd}
M.~Nabavi, M.H.K.~Siddiqui, and J.~Dargahi.
A new 9-point sixth-order accurate compact finite difference method for the Helmholtz equation.
\emph{Journal of Sound and Vibration}, 307:972--982, 2007.

\bibitem{pet}
P.G.~Petropoulos.
On the termination of the perfectly matched layer with local absorbing boundary conditions.
\emph{J. Comput. Phys.}, 143:665--673, 1998.

\bibitem{orsz}
S.~Orszag.
Spectral methods for problems in complex geometries.
\emph{J. Comp. Phys.}, 37:70--92, 1980.

\bibitem{s}
Y.~Saad.
\emph{Iterative Methods for Sparse Linear Systems}.
SIAM, 2nd edition, 2003.

\bibitem{sn}
A.A.~Samarskii and E.S.~Nikolaev.
\emph{Numerical Methods for Grid Equations}.
Birkhäuser, 1989.

\bibitem{sut}
G.~Sutmann.
Compact finite difference schemes of sixth order for the Helmholtz equation.
\emph{J. Comp. Appl. Math.}, 203:15--31, 2007.

\bibitem{st}
I.~Singer and E.~Turkel.
\emph{J. Comput. Acoust.}, 14(3):339--351, 2006.

\bibitem{umo}
N.~Umetani, S.P.~MacLachlan, and C.W.~Oosterlee.
\emph{Numer. Linear Algebra Appl.}, 8:603--626, 2009.

\bibitem{zs}
Y.~Zhuang and X.-H.~Sun.
A high-order fast direct solver for singular Poisson equations.
\emph{J. Comput. Phys.}, 171:79--94, 2001.

\bibitem{Lele92}
S.K.~Lele.
Compact finite difference schemes with spectral-like resolution.
\emph{Journal of Computational Physics}, 103:16--42, 1992.

\bibitem{mp}
T.A.~Manteuffel and S.V.~Parter.
Preconditioning and boundary conditions.
\emph{SIAM J. Numer. Anal.}, 27:656--694, 1998.

\bibitem{TW}
J.~Toivanen and M.~Wolfmayr.
A fast Fourier transform based direct solver for the Helmholtz problem.
\emph{Numer. Linear Algebra Appl.}, 27(3):e2283, 2020.

\bibitem{tggt}
E.~Turkel, D.~Gordon, R.~Gordon, and S.~Tsynkov.
Compact 2D and 3D sixth order schemes for the Helmholtz equation with variable wave number.
\emph{Journal of Computational Physics}, pages 272--287, 2012.

\bibitem{Gordon}
D.~Gordon and R.~Gordon.
Solution methods for linear systems with large off-diagonal elements and discontinuous coefficients.
\emph{Comp. Model. Eng. Sci.}, 53:23--45, 2009.

\bibitem{FFTW_doc}
M.~Frigo and S.~Johnson.
\emph{FFTW Manual}.
Massachusetts Institute of Technology, January 2003.

\bibitem{conv_diff}
J.~Kalita, A.~Dass, and D.~Dalal.
A transformation-free HOC scheme for steady convection-diffusion on non-uniform grids.
\emph{Int. J. Numer. Meth. Fluids}, 44:33--53, 2004.

\bibitem{chen2021complex}
H.~Chen, G.~Evéquoz, and T.~Weth.
Complex solutions and stationary scattering for the nonlinear Helmholtz equation.
\emph{SIAM Journal on Mathematical Analysis}, 53(2):2349--2372, 2021.

\bibitem{galkowski2025numerical}
J.~Galkowski and E.A.~Spence.
Numerical analysis of the high-frequency Helmholtz equation using semiclassical analysis.
\emph{arXiv preprint arXiv:2511.15287}, 2025.

\bibitem{lucido2024helmholtz}
M.~Lucido, G.A.~Casula, G.~Chirico, M.D.~Migliore, D.~Pinchera, and F.~Schettino.
Helmholtz--Galerkin technique in dipole field scattering from buried zero-thickness perfectly electrically conducting disk.
\emph{Applied Sciences}, 14(13):5544, 2024.

\bibitem{matsumoto2024fast}
Y.~Matsumoto.
Fast wavefield evaluation method based on modified proxy-surface-accelerated interpolative decomposition for two-dimensional scattering problems.
\emph{JSIAM Letters}, 16:61--64, 2024.

\bibitem{juraev2023applications}
D.A.~Juraev, P.~Agarwal, E.E.~Elsayed, and N.~Targyn.
Applications of the Helmholtz equation.
\emph{Advanced Engineering Days (AED)}, 8:28--30, 2023.

\bibitem{juraev2024helmholtz}
D.A.~Juraev, P.~Agarwal, E.E.~Elsayed, and N.~Targyn.
Helmholtz equations and their applications in solving physical problems.
\emph{Advanced Engineering Science}, 4:54--64, 2024.

\bibitem{klibanov2015two}
M.V.~Klibanov and V.G.~Romanov.
Two reconstruction procedures for a 3D phaseless inverse scattering problem for the generalized Helmholtz equation.
\emph{Inverse Problems}, 32(1):015005, 2015.

\bibitem{kow2024minimization}
P.-Z.~Kow, M.~Salo, and H.~Shahgholian.
A minimization problem with free boundary and its application to inverse scattering problems.
\emph{Interfaces Free Bound.}, 26(3):415--471, 2024.

\bibitem{yang2023novel}
A.L.~Yang.
A novel deep neural network algorithm for the Helmholtz scattering problem in the unbounded domain.
\emph{International Journal of Numerical Analysis \& Modeling}, 20(5), 2023.

\end{thebibliography}

\end{document}